\documentclass{amsart}
\usepackage{latexsym, amssymb, amsmath}
\def\R{\mathbb R}

\def\CP{\mathbb {CP}}

\newtheorem{thm}{Theorem}[section]
\newtheorem*{thmm}{Theorem~RSK}
\newtheorem*{thmmm}{Uncertainty Principle Lemma} 
\newtheorem*{thmmmm}{Hardy's Inequality} 
\newtheorem{lemm}[thm]{Lemma}

\newtheorem{prop}[thm]{Proposition}
\theoremstyle{remark}
\newtheorem{rmk}[thm]{\bf Remark}
\theoremstyle{definition}

\title{The uncertainty principle lemma under gravity\\ and\\ the discrete spectrum of Schr\"odinger operators} 

\author{Kazuo Akutagawa${}^*$}
\email{akutagawa$_{-}$kazuo@ma.noda.tus.ac.jp}
\address{Department of Mathematics, Tokyo University of Science, 
Noda 278-8510, Japan}

\author{Hironori Kumura${}^{\dagger}$}
\email{smhkumu@ipc.shizuoka.ac.jp} 
\address{Department of Mathematics, Shizuoka University, 
Shizuoka 422-8529, Japan} 

\thanks{${}^*$\ 
supported in part by the Grants-in-Aid for Scientific Research (C), 
Japan Society for the Promotion of Science, No.~18540098.\\ 
\quad\ ${}^{\dagger}$\ 
supported in part by the Grants-in-Aid for Scientific Research (C), 
Japan Society for the Promotion of Science, No.~18540212.} 

\date{January~8, 2009.}
\begin{document} 
\maketitle
\markboth{The uncertainty principle lemma under gravity}
{Kazuo Akutagawa and Hironori Kumura} 

\begin{abstract}
The {\it uncertainty principle lemma} for the Laplacian $\Delta$ on $\mathbb{R}^n$ shows 
the borderline-behavior of a potential $V$ for the following question~:~whether 
the Schr\"odinger operator $- \Delta + V$ has a finite or infinite number of the discrete spectrum. 
In this paper, we will give a generalization of this lemma on $\mathbb{R}^n$ 
to that on large classes of complete noncompact manifolds. 
Replacing $\mathbb{R}^n$ by some specific classes of complete noncompact manifolds, 
including hyperbolic spaces, 
we also establish some criterions for the above-type question. 
\end{abstract}
\maketitle 
\quad \\ 
  
\section{Introduction and Main Results} 

Let us start with the following crucial result:  
\begin{thmm}[{\rm Reed-Simon}~\cite{Re-Si-book-4}, {\rm Kirsch-Simon}~\cite{Ki-Si}]\label{Thm-RSK}\ \ 
Let $- \Delta + V$ be the Schr\"odinger operator with a potential $V \in C^0(\mathbb{R}^n)$ 
on $L^2(\mathbb{R}^n)$, where $n \geq 3$. 
Assume that $(- \Delta + V)|_{C^{\infty}_c(\mathbb{R}^n)}$ is essentially self-adjoint 
and that $\sigma_{{\rm ess}}(- \Delta + V) = [0, \infty)$, 
where $\sigma_{{\rm ess}}(- \Delta + V)$ denotes the essential spectrum of $- \Delta + V$. \\ 
$({\rm i})$\quad Assume that, there exists $R_0 > 0$ such that $V$ satisfies 
$$ 
V(x) \geq - \frac{(n-2)^2}{4r^2}\qquad {\rm for}\quad r := |x| \geq R_0. 
$$ 
Then, the set $\sigma_{{\rm disc}}(- \Delta + V)$ of the discrete spectrum is finite. \\ 
$({\rm ii})$\quad Assume that, there exist $\delta > 0$ and $R_1 > 0$ such that $V$ satisfies 
$$ 
V(x) \leq - (1+\delta)\frac{(n-2)^2}{4r^2}\qquad {\rm for}\quad r \geq R_1. 
$$ 
Then, $\sigma_{{\rm disc}}(- \Delta + V)$ is infinite.  
\end{thmm} 

Here, we recall the following fundamental two facts: 
First, $\sigma_{\rm ess}(- \Delta) = [0, \infty)$.   
Second, if $V \in C^0(\mathbb{R}^n)$ and 
$V(x) \rightarrow 0$ as $|x| \rightarrow \infty$, 
then $(- \Delta + V)|_{C^{\infty}_c(\mathbb{R}^n)}$ is essentially self-adjoint 
and $\sigma_{\textrm{ess}}(- \Delta + V) = [0, \infty)$ (cf.~\cite{Gu-Si-book}).  
Hence, the assumptions for the potential $V$ in Theorem~RSK are reasonable 
if one has an interest in only the asymptotic behavior of $V$ near infinity 
for the question whether $- \Delta + V$ has a finite or infinite number of the discrete spectrum. 
With these understandings, Theorem~RSK gives a complete answer to this question. 
In order both to know the borderline-behavior of $V$ 
and to prove the finiteness of $\sigma_{{\rm disc}}(- \Delta + V)$, 
the {\it Uncertainty Principle Lemma} below is crucial, 
which is heavily related to the {\it Heisenberg Uncertainty Principle} (cf.~\cite{Re-Si-book-2, Gu-Si-book}). 

\begin{thmmm}\label{Thm-UPL}\ \ 
When $n \geq 3$, the following inequality holds $:$ 
\begin{equation}  
\int_{\mathbb{R}^n} |\nabla u|^2 dx \geq 
\frac{(n-2)^2}{4} \int_{\mathbb{R}^n} \frac{u^2}{r^2} dx\qquad 
{\rm for}\quad u \in C_c^{\infty}(\mathbb{R}^n), 
\end{equation}  
where $\nabla u$ denotes the gradient of $u$. 
\end{thmmm} 

In this paper, we will study the finiteness and infiniteness of the discrete spectrum of a Schr\"odinger operator 
{\it under gravity}, that is, 
a Schr\"odinger operator $- \Delta_g + V$ on a complete noncompact $n$-manifold $(M, g)$. 
Here, $- \Delta_g$ denotes the Laplace-Beltrami operator with respect to $g$ on $C^{\infty}(M)$ and $V \in C^0(M)$. 
Throughout this paper, without particular mention, we always assume that $V \in C^0(M)$ 
and that $(- \Delta_g + V)|_{C^{\infty}_c(M)}$ is essentially self-adjoint 
for the sake of simplicity. 
For our purpose, we first prove the following {\it uncertainty principle lemma under gravity}, 
that is, the one on a complete noncompact manifold with ends of a specific type. 
\begin{thm}[{\bf Uncertainty Principle Lemma under Gravity}]\label{Thm-UPLG}\ \ 
Let $(M, g)$ be a complete noncompact Riemannian $n$-manifold, where $n \geq 2$. 
Assume that one of ends of $M$, denoted by $E$, has a compact connected $C^{\infty}$ boundary $W := \partial E$ 
such that the outward normal exponential map 
$\exp_W : \mathcal{N}^+(W) \rightarrow E$ 
is a diffeomorphism $($see ${\rm Fig.~1})$, 
where 
$$ 
\mathcal{N}^+(W) := \big{\{} v \in TM|_W \ \big{|}\ v \ {\rm is\ outward\ normal\ to}\ W \big{\}}. 
$$ 
Assume also that 
the mean curvature $H_W$ of $W$ with respect to the inward unit normal vector 
is positive. 
Take a positive constant $R > 0$ satisfying 
$$ 
H_W \geq \frac{1}{R}\qquad {\rm on}\quad W,
$$ 
and set 
$$
\rho(x) := {\rm dist}_g(x, W),\quad r(x) := \rho(x) + R\qquad {\rm for}\quad x \in E.
$$   
Then, for all $u \in C_c^{\infty}(M)$, we have 
\begin{align} 
& \int_E |\nabla u|^2 dv_g \\ 
\geq \int_E \Big{\{} \frac{1}{4 r^2}& + \frac{1}{4}(\Delta_g r)^2 - \frac{1}{2}|\nabla d r|^2 
- \frac{1}{2} {\rm Ric}_g(\nabla r, \nabla r) \Big{\}} u^2 dv_g 
+ \frac{1}{2} \int_W \Big{(} \Delta_g r - \frac{1}{R} \Big{)} u^2 d\sigma_g \notag \\ 
\geq \int_E \Big{\{} \frac{1}{4 r^2}& + \frac{1}{4}(\Delta_g r)^2 - \frac{1}{2}|\nabla d r|^2 
- \frac{1}{2} {\rm Ric}_g(\nabla r, \nabla r) \Big{\}} u^2 dv_g, \notag 
\end{align} 
where ${\rm Ric}_g$, $dv_g$ and $d\sigma_g$ denote respectively the Ricci curvature, 
the Riemannian volume measure of $g$ and the $(n-1)$-dimensional Riemannian volume measure of $(W, g|_W)$. 
In particular, if $(M, g)$ has a pole $p_0 \in M$ $(${\rm cf.}~\cite{Gr-Wu-book}$)$, then 
\begin{equation} 
\int_M |\nabla u|^2 dv_g 
\geq \int_M \Big{\{} \frac{1}{4r^2} + \frac{1}{4}(\Delta_g r)^2 - \frac{1}{2}|\nabla dr|^2 
- \frac{1}{2} {\rm Ric}_g(\nabla r, \nabla r) \Big{\}} u^2 dv_g, 
\end{equation} 
where $r(x) := {\rm dist}_g(x, p_0)$ for $x \in M$. 
\end{thm} 
\begin{center}
\input{figends.tex}
Figure~1
\end{center}

\begin{rmk}\  
(i)\quad For $x_0 \in E$, 
$(\nabla dr)(x_0)$ and $(\Delta_g r)(x_0)$ are respectively 
the {\it second fundamental form} and the {\it mean curvature} 
of the level hypersurface $r^{-1}(x_0) = \{ x \in E\ |\ r(x) = r(x_0) \}$ at $x_0$ 
(with respect to the inward unit normal vector). \\ 
(ii)\quad Let $(M, g)$ be the Euclidean $n$-space $\mathbb{R}^n = (\mathbb{R}^n, g_0)$\ 
(resp.~the hyperbolic $n$-space $\mathbb{H}^n(-\kappa) = (\mathbb{H}^n(-\kappa), g_{\kappa})$ 
of constant negative curvature $- \kappa$). 
For each $R > 0$, we denote $B_R({\bf 0})$ the geodesic open ball of radius $R$ centered at the origin ${\bf 0}$ 
of $\mathbb{R}^n$\ (resp.~$\mathbb{H}^n(-\kappa)$ )
and 
$$
r(x) := \rho(x) + R = {\rm dist}_g(x, \partial B_R({\bf 0}) ) + R\ \ \big{(}\ = {\rm dist}_g(x, {\bf 0})\ \big{)} 
$$ 
for $x \in E := \mathbb{R}^n - B_R({\bf 0})$\ (resp.~$E := \mathbb{H}^n(-\kappa) - B_R({\bf 0})$\ ). 
Then, the term appearing in the boundary integral of (2) can be described as 
\begin{equation} 
\Big{(} \Delta_g r - \frac{1}{R} \Big{)}\Big{|}_{\partial B_R({\bf 0})} = 
\begin{cases} 
\ \ \frac{n-2}{R} & {\rm if} \quad (M, g) = \mathbb{R}^n, \\ 
\ \ (n-1)~\sqrt{\kappa}~{\rm coth}(\sqrt{\kappa} R) - \frac{1}{R} & {\rm if}\quad (M, g) = \mathbb{H}^n(-\kappa), \\  
\end{cases} 
\end{equation} 
and hence, when $n \geq 2$, this term is non-negative for all $R > 0$ in the both cases. 
More generally, let $(M, g)$ be a complete $n$-manifold with the following type end: 
$$ 
([0, \infty) \times N, dr^2 + f(r)^2\cdot g_N)\quad {\rm with}\quad 
f'(r) > 0 \quad {\rm for}\quad r \geq R_0 > 0,  
$$ 
where $(N, g_N)$ is a closed Riemannian $(n-1)$-manifold and $R_0 > 0$ is some positive constant. 
The condition $f'(r) > 0$ ($r \geq R_0$) implies that 
the boundary $\partial E_L$ of each $E_L := [L, \infty) \times N$ has the positive mean curvature 
$H_{\partial E_L} = \frac{(n-1) f'(L)}{f(L)} > 0$ provided that $L \geq R_0$. 
Then, if one choose $R > 0$ appropriately corresponding to $E_L$, 
the term $\Big{(} \Delta_g r - \frac{1}{R} \Big{)}\Big{|}_{\partial E_L}$  is non-negative. 
However, for a general end $E$ and a sequence $\{L_i\}$ with $0 < L_1 < L_2 < \cdots < L_i < \cdots \nearrow \infty$, 
the non-negativity of 
$\Big{(} \Delta_g r - \frac{1}{R} \Big{)}\Big{|}_{\partial E_{L_i}}$  
does not necessarily require an expansion of $g$ on $E$.  
For instance, suppose that $(M, g)$ has the following periodic end: 
$$ 
\big{(} [0, \infty) \times N, dr^2 + (2 + \sin r)^2 g_N \big{)}.  
$$ 
Then, for $E_{2m\pi} = [2m\pi, \infty) \times~N$ 
$(m \geq \frac{1}{(n-1)\pi}, m \in \mathbb{Z} )$ and $R := 2m\pi$, 
we obtain 
$$ 
\Big{(} \Delta_g r - \frac{1}{R} \Big{)}\Big{|}_{\partial E_{2m\pi}} = \frac{n-1}{2} -\frac{1}{2m\pi} \geq 0.  
$$ 
(iii)\quad When $(M, g) = (\mathbb{R}^n, g_0)$, 
we have 
$$ 
\Delta r = \frac{n-1}{r},\qquad |\nabla dr|^2 = \frac{n-1}{r^2},\qquad {\rm Ric}_{g_0}(\nabla r, \nabla r) = 0.  
$$ 
Hence, the inequality~(3) includes the inequality~(1) 
in the original uncertainty principle lemma on $\mathbb{R}^n$.  
\end{rmk} 

Applying Theorem~\ref{Thm-UPLG} to the hyperbolic $n$-space $\mathbb{H}^n(-\kappa)$, 
we obtain the explicit result below. 
\begin{thm}[{\bf Uncertainty Principle Lemma on $\mathbb{H}^n(-\kappa)$}]\label{Thm-UPLH}\ \ 
When $n \geq 2$, for $u \in C_c^{\infty}(\mathbb{H}^n(-\kappa))$ and $R > 0$, the following holds  
\begin{align} 
& \int_{\mathbb{H}^n(-\kappa) - B_R({\bf 0})} |\nabla u|^2 dv_{g_{\kappa}} \\ 
\geq & \int_{\mathbb{H}^n(-\kappa) - B_R({\bf 0})} 
\Big{\{} \frac{(n-1)^2 \kappa}{4} + \frac{1}{4 r^2} + \frac{(n-1)(n-3) \kappa}{4~{\rm sinh}^2(\sqrt{\kappa} r)} \Big{\}} 
u^2 dv_{g_{\kappa}}. \notag  
\end{align} 
In particular, 
\begin{align} 
\int_{\mathbb{H}^n(-\kappa)} |\nabla u|^2 dv_{g_{\kappa}} 
\geq \int_{\mathbb{H}^n(-\kappa)} 
\Big{\{} \frac{(n-1)^2 \kappa}{4} + \frac{1}{4 r^2} + \frac{(n-1)(n-3) \kappa}{4~{\rm sinh}^2(\sqrt{\kappa} r)} \Big{\}} 
u^2 dv_{g_{\kappa}}.
\end{align}  
Here, $r(x) := {\rm dist}_{g_{\kappa}}(x, {\bf 0})$. 
\end{thm} 
\begin{rmk}\ \ 
Since 
$$ 
\lim_{\kappa \searrow 0} \frac{{\rm sinh}(\sqrt{\kappa} r)}{\sqrt{\kappa} r} = 1, 
$$  
by letting $\kappa \searrow 0$ in $(6)$, 
we recover the inequality~(1) in the original uncertainty principle lemma on $\mathbb{R}^n$. 
\end{rmk} 

Using this lemma, we also obtain the following explicit criterion. 
\begin{thm}\label{Thm-H}\ \ 
Let $- \Delta_{g_{\kappa}} + V$ be the Schr\"odinger operator with a potential $V$ 
on $L^2(\mathbb{H}^n(-\kappa))$, where $n \geq 2$.  
Assume that $\sigma_{{\rm ess}}(- \Delta_{g_{\kappa}} + V) = [\frac{(n-1)^2 \kappa}{4}, \infty)$. \\  
$({\rm i})$\quad Assume that, there exist $R_0 > 0$ such that $V$ satisfies 
$$ 
V(x) \geq - \Big{(} \frac{1}{4 r^2} + \frac{(n-1)(n-3) \kappa}{4~{\rm sinh}^2(\sqrt{\kappa} r)} \Big{)}\qquad 
{\rm for}\quad r \geq R_0. 
$$ 
Then, $\sigma_{{\rm disc}}(- \Delta_{g_{\kappa}} + V)$ is finite. \\ 
$({\rm ii})$\quad Assume that, there exist $\delta > 0$ and $R_1 > 0$ such that $V$ satisfies 
$$ 
V(x) \leq - (1+\delta)\frac{1}{4r^2}\qquad {\rm for}\quad r \geq R_1. 
$$ 
Then, $\sigma_{{\rm disc}}(- \Delta_{g_{\kappa}} + V)$ is infinite.  
\end{thm} 

We regards Theorem~RSK and Theorem~\ref{Thm-H} as the model criterions. 
Using also Theorem~\ref{Thm-UPLG} in a sense of approximation, 
we have the following criterions. 
\begin{thm}\label{Thm-ALE}\ \ 
Let $(M, g)$ be a complete noncompact $n$-manifold with $n \geq 3$ 
and $- \Delta_g + V$ the Schr\"odinger operator with a potential $V$. 
Take a point $p_0 \in M$, and set $r(p) := {\rm dist}_g(p, p_0)$ for $p \in M$. 
Assume that, 
there exist some positive constants $L, L', K > 0$ 
and a small positive constant $\tau~(0 < \tau < 1)$ such that 
\begin{align*} 
&(1.1)\qquad \qquad |\mathcal{R}_g| \leq L r^{-(2+\tau)},\quad 
|\nabla {\rm Ric}_g| \leq L' r^{-(3+\tau)}\quad {\rm for\ all\ large}\ r, 
\qquad \qquad \qquad \qquad \qquad \qquad \qquad \qquad \\ 
&(1.2)\qquad \qquad {\rm Vol}(B_t(p_0)) \geq K t^n\quad {\rm for\ all}\ t > 0, \\ 
&(1.3)\qquad \qquad \sigma_{\rm ess}(- \Delta_g + V) = [0, \infty), 
\end{align*}
where $\mathcal{R}_g$ denotes the Riemannian curvature tensor of $g$. \\ 
$({\rm i})$\quad Assume that, there exist $\delta_0 > 0$ and $R_0 > 0$ such that $V$ satisfies 
$$ 
V(x) \geq - (1 - \delta_0) \frac{(n-2)^2}{4r^2}\qquad {\rm for}\quad r \geq R_0. 
$$ 
Then, $\sigma_{{\rm disc}}(- \Delta_g + V)$ is finite. \\ 
$({\rm ii})$\quad Assume that, there exist $\delta_1 > 0$ and $R_1 > 0$ such that $V$ satisfies 
$$ 
V(x) \leq - (1 + \delta_1)\frac{(n-2)^2}{4r^2}\qquad {\rm for}\quad r \geq R_1. 
$$ 
Then, $\sigma_{{\rm disc}}(- \Delta_g + V)$ is infinite.   
\end{thm} 
\begin{rmk}\ \ 
By Theorem~(1.1) and Remark~(1.8) in \cite{Ba-Ka-Na}, 
$(M, g)$ is {\it asymptotically locally Euclidean} (abbreviated to {\it ALE}~) 
of order $\tau$ with finitely many ends. 
Namely, there exists a relatively compact open set $\mathcal{O}$ such that 
each component (i.e., each end) of $M - \mathcal{O}$ has {\it coordinates at infinity} 
$x = (x^1, \cdots, x^n)$; that is, 
there exist $R > 0$ and a finite subgroup $\Gamma \subset O(n)$ 
acting freely on $\mathbb{R}^n - B_R({\bf 0})$ 
such that the end is diffeomorphic to $\big{(} \mathbb{R}^n - B_R({\bf 0}) \big{)} \big{/} \Gamma$ 
and, with respect to $x = (x^1, \cdots, x^n)$, the metric $g$ satisfies the following on the end:  
$$ 
g_{ij} = \delta_{ij} + O(|x|^{-\tau}),\qquad 
\partial_k g_{ij} = O(|x|^{-1-\tau}),\qquad 
\partial_k \partial_{\ell}g_{ij} = O(|x|^{-2-\tau}).  
$$ 
In this case, one can easily check that $\sigma_{\rm ess}(- \Delta_g) = [0, \infty)$ for the ALE manifold $(M, g)$. 
From the argument in Proof of Theorem~\ref{Thm-ALE}-(ii) later, 
one can also easily see that, 
if the potential $V$ satisfies the decay condition in (ii) only on one end, 
then $\sigma_{\rm disc}(- \Delta_g + V)$ is infinite. 
\end{rmk} 
\begin{thm}\label{Thm-AH}\ \ 
Let $(M, g)$ be an asymptotically hyperbolic $n$-manifold of class $C^2$ 
with $n \geq 2$, 
and $- \Delta_g + V$ the Schr\"odinger operator with a potential $V$. 
Assume that 
$$ 
\sigma_{\rm ess}(- \Delta_g + V) = [\frac{(n-1)^2}{4}, \infty). 
$$ 
Take a point $p_0 \in M$, and set $r(x) := {\rm dist}_g(x, p_0)$ for $x \in M$. \\ 
$({\rm i})$\quad Assume that, there exist $\delta_0 > 0$ and $R_0 > 0$ such that $V$ satisfies 
$$ 
V(x) \geq - (1 - \delta_0) \frac{1}{4r^2}\qquad {\rm for}\quad r \geq R_0. 
$$ 
Then, $\sigma_{{\rm disc}}(- \Delta_g + V)$ is finite. \\ 
$({\rm ii})$\quad Assume that, there exist $\delta_1 > 0$ and $R_1 > 0$ such that $V$ satisfies 
$$ 
V(x) \leq - (1 + \delta_1)\frac{1}{4r^2}\qquad {\rm for}\quad r \geq R_1. 
$$ 
Then, $\sigma_{{\rm disc}}(- \Delta_g + V)$ is infinite.   
\end{thm} 
\begin{rmk}\ \ 
Let $\overline{M}$ be a compact $C^{\infty}$ $n$-manifold with boundary (possibly disconnected) 
and $M$ its interior. 
Then, a metric $g$ on $M$ is said to be   
{\it asymptotically hyperbolic of class} $C^2$ 
if $g$ satisfies the following: 
There exists a {\it defining function} $\lambda \in C^{\infty}(\overline{M})$ of $\partial \overline{M}$ 
such that the conformally rescaled metric $\overline{g} := \lambda^2 g$ has a $C^2$-extention on $\overline{M}$  
and that $|d \lambda|^2_{\overline{g}} = 1$ on $\partial \overline{M}$ (cf.~\cite{Ma, Gr-Le, Le-book}). 
R.~Mazzeo~\cite{Ma} proved that such $(M, g)$ is complete 
and has sectional curvatures uniformly approaching $- 1$ near $\partial M$. 
Moreover, 
$\sigma_{\rm ess}(- \Delta_g) = [\frac{(n-1)^2}{4}, \infty)$ \ (cf.~\cite{An, Ku}). 
\end{rmk}

In the next section, we first introduce Hardy's inequality. 
We then prove Theorem~\ref{Thm-UPLG}. 
Theorem~\ref{Thm-UPLG}, that is, the {\it uncertainty principle lemma under gravity} 
arises from essentially this inequality. 
Using this theorem, we also prove the other main results mentioned above, 
including Theorem~RSK with a simple proof. 
In Section~3, we give one more application of Thorem~1.1 
and an interesting example of a noncompact manifold on which this theorem holds. \\  
{\bf Acknowledgements.} 
The authors would like to thank Hideo Tamura, Hiroshi Isozaki, Rafe Mazzeo and Richard Schoen for helpful discussions.  
\quad \\ 

\section{Hardy's Inequality and Proof of Main Results} 

We first recall the following Hardy's inequality~\cite[Theorem~327]{Ha-Li-Po-book}, 
which is an essential source of the inequalities in (1) and (2). 

\begin{thmmmm}\label{Thm-Hardy}\ \ 
For any $f \in C_c^1\big{(}[0, \infty)\big{)}$ and $R > 0$, the following holds. 
\begin{equation}
\int_R^{\infty} |f'(t)|^2 dt \geq \int_R^{\infty} \frac{f(t)^2}{4 t^2} dt - \frac{f(R)^2}{2 R}. 
\end{equation} 
\end{thmmmm} 
\quad \\ 
{\it Proof of Theorem~\ref{Thm-UPLG}}. 
Let $y = (y^2, \cdots, y^n)$ be local coordinates on an open subset $U$ of $W$. 
Under the identification $E \cong \mathcal{N}^+(W)$ by the outward normal exponential map, 
then the metric $g$ can be described as 
$$
g(r, y) = dr^2 + g^W_{\alpha \beta}(r, y)dy^{\alpha}dy^{\beta}\quad 
{\rm on}\quad [R, \infty) \times U\qquad (2 \leq \alpha, \beta \leq n), 
$$ 
where $(\rho, y)$ are the {\it Fermi coordinates} on $[0, \infty) \times U$ and $r := \rho + R$.  
For each $u \in C_c^{\infty}(M)$, we have on $[R, \infty) \times \{y\}$ \ \ ($y \in U$) 
\begin{align} 
& \int_R^{\infty} |\partial_r u|^2 \sqrt{g^W} dr \\ 
= \int_R^{\infty}&|\partial_r(\sqrt{g^W}^{\frac{1}{2}}u)|^2 dr 
+ \int_R^{\infty} \Big{\{} \frac{1}{4}(\Delta_g r)^2 - \frac{1}{2}|\nabla dr|^2 
- \frac{1}{2}{\rm Ric}_g(\nabla r, \nabla r) \Big{\}} u^2 \sqrt{g^W} dr \notag \\ 
&\qquad \qquad \qquad \quad \ + \frac{1}{2} \Big{(} \Delta_g r\cdot u^2\cdot \sqrt{g^W} \Big{)} \Big{|}_{(R, y)} \notag. 
\end{align} 
where $\sqrt{g^W} := \sqrt{\textrm{det}(g^W_{\alpha \beta}(r, y))}$. 

The proof of the identity (8) is the below: 
A direct computation shows that on the interval $[R, \infty) \times \{y\}$
\begin{align}
& \int_R^{\infty} |\partial_r(\sqrt{g^W}^{\frac{1}{2}}u)|^2 dr \\ 
= & \ \frac{1}{4} \int_R^{\infty} \frac{\big{(}\partial_r\sqrt{g^W}\big{)}^2}{\sqrt{g^W}} u^2 dr 
+ \frac{1}{2} \int_R^{\infty} \big{(}\partial_r\sqrt{g^W}\big{)} \partial_r (u^2) dr 
+ \int_R^{\infty} \sqrt{g^W} |\partial_r u|^2 dr. \notag 
\end{align}  
Using integration by parts, we can calculate 
the second term of the right hand side of the above as 
\begin{align} 
\frac{1}{2} \int_R^{\infty} \big{(}\partial_r\sqrt{g^W}\big{)} \partial_r (u^2)& dr 
= \frac{1}{2} \big{[} \big{(}\partial_r\sqrt{g^W}\big{)} u^2 \big{]}_{r = R}^{r = \infty} 
- \frac{1}{2} \int_R^{\infty} \big{(}\partial^2_r\sqrt{g^W}\big{)} u^2 dr \\ 
& = - \frac{1}{2} \int_R^{\infty} \big{(}\partial^2_r\sqrt{g^W}\big{)} u^2 dr 
- \frac{1}{2} \Big{(} \big{(}\partial_r\sqrt{g^W}\big{)} u^2 \Big{)} \Big{|}_{(R, y)}. \notag 
\end{align}  
Since $\Delta_g r = \frac{\partial_r\sqrt{g^W}}{\sqrt{g^W}}$ and 
$- \partial_r (\Delta_g r) = |\nabla dr|^2 + \textrm{Ric}_g(\nabla r, \nabla r)$, 
we now get the following two identities 
\begin{align*} 
& \partial_r^2 \sqrt{g^W} = - \big{\{} |\nabla dr|^2 + {\rm Ric}_g(\nabla r, \nabla r) \big{\}} \sqrt{g^W} 
+ (\Delta_g r)^2 \sqrt{g^W}, \\ 
& \frac{\big{(}\partial_r \sqrt{g^W}\big{)}^2}{\sqrt{g^W}} = (\Delta_g r)^2 \sqrt{g^W}. 
\end{align*} 
These two identities combined with (9) and (10) imply the identity (8). 

Hardy's inequality (7) implies the following 
\begin{equation}  
\int_R^{\infty} |\partial_r \big{(} \sqrt{g^W}^{\frac{1}{2}} u \big{)} |^2 dr 
\geq \int_R^{\infty} \frac{u^2}{4 r^2} \sqrt{g^W} dr 
- \frac{1}{2} \Big{(} \frac{u^2}{R}\sqrt{g^W} \Big{)} \Big{|}_{(R, y)}.  
\end{equation}  
Note also that, in terms of the coordinates $(r, y)$ on 
$[R, \infty) \times U \ (\ \subset [R, \infty) \times W\ )$,  
the volume element $dv_g$ can be expressed as 
\begin{equation*} 
dv_g(r, y) = dr~d\sigma_g(r, y) = \sqrt{g^W(r, y)}~dr~dy^2 \cdots dy^n. 
\end{equation*} 
Substituting the Hardy's inequality (11) into the identity (8) 
combined with $|\nabla u|^2 \geq |\partial_r u|^2$ 
and integrating its both sides over $W$ 
(locally with respect to $dy^2 \cdots dy^n$), 
we then get the inequality (2). 

Now, we assume that $(M, g)$ has a pole $p_0 \in M$. 
For each small $\varepsilon > 0$, 
set $E := M - B_{\varepsilon}(p_0)$ in the inequality (2), 
where $B_{\varepsilon}(p_0)$ denotes the open geodesic ball of radius $\varepsilon$ centered at $p_0$. 
Letting $\varepsilon \searrow 0$ in the integration over $W = \partial B_{\varepsilon}(p_0)$ of (2), 
we then have 
$$ 
\frac{1}{2} \int_{\partial B_{\varepsilon}(p_0)} 
\Big{(} \Delta_g r - \frac{1}{\varepsilon} \Big{)} u^2 d\sigma_g 
= \frac{n - 2}{2}~\sigma_{n-1} \cdot u(p_0)^2\cdot \varepsilon^{n-2} + O(\varepsilon^{n-1}) \rightarrow 0, 
$$ 
where $\sigma_{n-1}$ denotes the $(n-1)$-dimensional volume $\textrm{Vol}(S^{n-1}(1))$ of 
the unit $(n-1)$-sphere $S^{n-1}(1)$ of $\mathbb{R}^n$. 
Combining this with (2), we obtain the inequality (3). \qed \\ 
\quad \\ 
{\it Proof of Theorem~\ref{Thm-UPLH}}. 
\ \ On the space $\mathbb{H}^n(- \kappa)$, we have 
\begin{align*} 
& \Delta_{g_\kappa} r = (n-1)\sqrt{\kappa}~{\rm coth} (\sqrt{\kappa} r),\qquad  
|\nabla dr|^2 = (n-1)\kappa~{\rm coth}^2 (\sqrt{\kappa} r), \\ 
& {\rm Ric}_{g_{\kappa}}(\nabla r, \nabla r) = - (n-1) \kappa. 
\end{align*} 
Applying the inequalities (2) and (3) to the space $\mathbb{H}^n(- \kappa)$ 
combined with these identities and (4), 
we obtain the inequalities (5) and (6). \qed \\  

For both the self-containedness and the later use on Proofs of Theorems~\ref{Thm-H}, \ref{Thm-ALE}, \ref{Thm-AH}, 
we give here a simple proof of Theorem~RSK, particularly the finiteness assertion (i) 
by using the inequality (2), not the original inequality (1). 
For the proof of the infiniteness assertion~(ii), 
we also give a unified view of the mechanism for constructing a {\it nice test function} 
on each rotationally symmetric Riemannian $n$-manifold $(\mathbb{R}^n, dr^2 + h(r)^2\cdot g_{S^{n-1}(1)})$. 
Here, $g_{S^{n-1}(1)}$ denotes the standard metric of constant curvature $1$ on $S^{n-1}(1)$. \\ 
\quad \\ 
{\it Proof of Theorem~RSK}.\ \ \underline{Assertion~(i)}. 
\ \ Suppose that $\sigma_{\textrm{disc}}(- \Delta + V)$ is infinite. 
Then, there exist a family $\{\lambda_i\}_{i=1}^{\infty}$ of {\it negative} eigenvalues 
and the corresponding eigenfunctions $\{\varphi_i\}_{i=1}^{\infty}$ satisfying 
$$ 
\lambda_1 < \lambda_2 \leq \lambda_3 \leq \cdots \leq \lambda_k \leq \cdots \nearrow 0,\qquad 
\int_{\mathbb{R}^n} \varphi_i \varphi_j dx = \delta_{ij}\quad {\rm for\ all}\quad i, j \geq 1. 
$$ 
Decompose $\mathbb{R}^n$ into the two pieces as 
$$ 
\mathbb{R}^n = \overline{B_{R_0}({\bf 0})}\ \cup\ \big{(} \mathbb{R}^n - B_{R_0}({\bf 0}) \big{)}. 
$$ 
We consider the Neumann eigenvalue problem for $- \Delta + V$ 
on both $\overline{B_{R_0}({\bf 0})}$ and $\mathbb{R}^n - B_{R_0}({\bf 0})$. 
Arrange {\it all} the eigenvalues of $\overline{B_{R_0}({\bf 0})}$ and $\mathbb{R}^n - B_{R_0}({\bf 0})$ 
in increasing order, with repetition according to multiplicity: 
$$ 
\mu_1 \leq \mu_2 \leq \cdots . 
$$ 
From {\it Domain monotonicity of eigenvalues} ({\it vanishing Neumann data}) (cf.~\cite{Ch-book}), 
we have 
\begin{align}  
\mu_i \leq \lambda_i\ (\ < 0\ )\qquad {\rm for\ all}\quad i \geq 1. 
\end{align} 

Now applying the inequality~(2) to $\mathbb{R}^n - B_{R_0}({\bf 0})$ combined with Remark~1.2-(ii), (iii), 
we then obtain 
$$   
\int_{\mathbb{R}^n - B_{R_0}({\bf 0})} |\nabla u|^2 dx 
\geq \frac{(n-2)^2}{4} \int_{\mathbb{R}^n - B_{R_0}({\bf 0})} \frac{u^2}{r^2} dx\qquad 
{\rm for}\quad u \in C_c^{\infty}(\mathbb{R}^n). 
$$ 
Therefore, the assumption for the potential $V$ implies 
$$ 
\int_{\mathbb{R}^n - B_{R_0}({\bf 0})} \big{(} |\nabla u|^2 + V u^2 \Big{)} dx 
\geq 0\qquad {\rm for}\quad u \in C_c^{\infty}(\mathbb{R}^n), 
$$ 
and hence the Schr\"odinger operator $- \Delta + V$ on $\mathbb{R}^n - B_{R_0}({\bf 0})$ 
(with Neumann boundary condition) 
has no negative eigenvalue. 
On the other hand, since $\overline{B_{R_0}({\bf 0})}$ is compact, 
$- \Delta + V$ on $\overline{B_{R_0}({\bf 0})}$ (with Neumann boundary condition) has 
only a finite number of negative eigenvalues. 
These facts combined with (12) contradict that 
$\sigma_{\textrm{disc}}(- \Delta + V)$ is infinite. \\ 
\underline{Assertion~(ii)}. 
\ \ First, we prove the following. 
\begin{lemm}\label{Sub}  
Let $A$ be a Schr\"odinger operator $- \frac{d^2}{dt^2} + V(t)$ 
with its domain $C_c^{\infty}(0, \infty)$ and $V \in C^0(0, \infty)$. 
Assume that $A$ has a self-adjoint extension $\widehat{A}$ on $L^2((0, \infty); dt)$ 
and that $\sigma_{\rm ess}(\widehat{A}) = [0, \infty)$. 
Assume also that, there exist $\widetilde{\delta} > 0$ and $\widetilde{R} > 0$ such that 
$V$ satisfies 
$$ 
V(t) \leq - (1 + \widetilde{\delta}) \frac{1}{4 t^2}\qquad{\rm for}\quad t \geq \widetilde{R}. 
$$ 
Then, $\sigma_{\rm disc}(\widehat{A})$ is infinite.  
\end{lemm} 
\noindent 
{\it Proof of Lemma~\ref{Sub}}.\ \ 
For a fixed $R \geq \widetilde{R}$, we define a cut-off function $\chi(t)$ as follows: 
$$ 
\chi(t) := 
\begin{cases} 
\ \ 0\quad & {\rm if}\quad t \in [0, R], \\ 
\ \ \frac{1}{R}(t - R)\quad & {\rm if}\quad t \in [R, 2R], \\ 
\ \ 1\quad & {\rm if}\quad t \in [2R, kR], \\ 
\ \ - \frac{1}{kR}(t - 2kR)\quad & {\rm if}\quad t \in [kR, 2kR], \\ 
\ \ 0\quad & {\rm if}\quad t \in [2kR, \infty), \\ 
\end{cases} 
$$ 
where $k$ is a large positive constant defined later. 
Set $\varphi(t) := \chi(t) t^{\frac{1}{2}}$ for $t > 0$. 
Then, the direct computation shows that 
$$ 
|\varphi'(t)|^2 + V(t) \varphi(t)^2 
= |\chi'(t)|^2 t + \Big{(} V(t) + \frac{1}{4 t^2} \Big{)} \varphi(t)^2 + \frac{1}{2} (\chi(t)^2)'. 
$$ 
Integrating the both sides over $(0, \infty)$, we have 
\begin{align*} 
\int_0^{\infty} \{ |\varphi'|^2 + V \varphi^2 \} dt 
= &\ \int_0^{\infty} |\chi'(t)|^2 t dt + \int_0^{\infty} \Big{(} V + \frac{1}{4 t^2} \Big{)} \varphi^2 dt \\  
\leq &\ \int_0^{\infty} |\chi'(t)|^2 t dt - \frac{\widetilde{\delta}}{4} \int_0^{\infty} \frac{\chi^2}{t} dt \\ 
\leq &\ \frac{1}{R^2} \int_R^{2R} t dt + \frac{1}{(kR)^2} \int_{kR}^{2kR} t dt 
- \frac{\widetilde{\delta}}{4} \int_{2R}^{kR} \frac{\chi^2}{t} dt \\ 
= &\ 3 - \frac{\widetilde{\delta}}{4} \log\Big{(} \frac{k}{2} \Big{)}. 
\end{align*} 
Hence, there exists a large positive constant $k_0 = k_0(\delta, n)$ such that 
$$ 
\int_0^{\infty} \{ |\varphi'|^2 + V \varphi^2 \} dt < 0\qquad {\rm if}\quad k \geq k_0. 
$$ 

Successively choosing $R$ and $k$, 
we then get a sequence $\{\varphi_i\}_{i=1}^{\infty}$ of functions in $W^{1,2}((0,\infty))$ 
with compact support such that 
\begin{align*}  
&\ \ {\rm supp}~\varphi_i~\cap~{\rm supp}~\varphi_j = \emptyset\qquad {\rm if}\quad i \ne j, \\ 
&\int_0^{\infty} \{ |\varphi_i'|^2 + V \varphi_i^2 \} dt < 0\qquad {\rm for\ all}\quad i \geq 1. 
\end{align*} 
Therefore, the min-max principle implies that $\sigma_{\textrm{disc}}(\widehat{A})$ is infinite. \qed \\ 

Next, we consider a rotationally symmetric Riemannian $n$-manifold 
\begin{align} 
(M, g) := (\mathbb{R}^n, dr^2 + h(r)^2\cdot g_{S^{n-1}(1)}). 
\end{align}  
The restriction $(- \Delta_g)|_{\rm radial}$ of $- \Delta_g$ to the space of radial functions on $M$ 
is given by 
$$ 
(- \Delta_g)|_{\rm radial} = - \frac{d^2}{dr^2} - (n-1) \frac{h'(r)}{h(r)} \frac{d}{dr}\qquad 
{\rm on}\quad L^2((0, \infty); h(r)^{n-1} dr).  
$$ 
Now we define a {\it unitary} operator $U$ as 
$$ 
U : L^2((0, \infty); h(r)^{n-1} dr) \ni \psi \longmapsto h^{\frac{n-1}{2}} \psi \in L^2((0, \infty); dr). 
$$
Then, the self-adjoint operator $(- \Delta_g)|_{\rm radial}$ on $L^2((0, \infty); h(r)^{n-1} dr)$ 
is unitary equivalent to 
the self-adjoint operator $L_g := U \circ (- \Delta_g)|_{\rm radial} \circ U^{-1}$ on $L^2((0, \infty); dr)$. 
Remark also that the multiplication operator $V : \psi \mapsto V\cdot \psi$ 
is conserved under this unitary transformation, that is, $U \circ V \circ U^{-1} = V$. 
For $u \in {\rm Dom}(L_g)$, we can calculate $L_gu$ explicitly as 
\begin{align*} 
L_gu &= \Big{(} h^{\frac{n-1}{2}}\circ (- \Delta_g)|_{\rm radial} \Big{)} (h^{-\frac{n-1}{2}} u) \\ 
&= - \frac{d^2u}{dr^2} + \Big{\{} \frac{(n-1)(n-3)}{4} \Big{(}\frac{h'}{h}\Big{)}^2 + \frac{n-1}{2} \frac{h''}{h} \Big{\}} u. \notag 
\end{align*} 

With these understandings, we will now return the construction of a desired test function on $\mathbb{R}^n$. 
Note that, since the upper bound for the potential $V(x)$ on $\{r = |x| \geq R_1\}$ 
is given by the radial function $- (1+\delta)\frac{(n-2)^2}{4 r^2}$, 
hence we may assume that $V$ is also a radial function on $\mathbb{R}^n$. 
So, we will construct our desired test function as a radial function. 
Since the Euclidian $n$-space $\mathbb{R}^n$ is given by (13) with $h(r) := r$, 
the potential term of $L_g$ is 
$$ 
\frac{(n-1)(n-3)}{4} \Big{(}\frac{h'}{h}\Big{)}^2 + \frac{n-1}{2} \frac{h''}{h} 
= \frac{(n-1)(n-3)}{4 r^2},  
$$ 
and hence the potential of the operator $U \circ \Big{(} (- \Delta_g)|_{\rm radial} + V(r) \Big{)} \circ U^{-1}$ 
is given by  
$$ 
\frac{(n-1)(n-3)}{4 r^2} + V(r).  
$$ 
If we set $\widetilde{\delta} := (n-2)^2 \delta > 0$ in Lemma~\ref{Sub}, 
the condition for the potential of the operator $U \circ \Big{(} (- \Delta_g)|_{\rm radial} + V(r) \Big{)} \circ U^{-1}$ 
$$ 
\frac{(n-1)(n-3)}{4 r^2} + V(r) \leq - (1 + \widetilde{\delta}) \frac{1}{4 r^2}\qquad {\rm for}\quad r \geq R_1 
$$  
is equivalent to that for the potential $V(r)$ 
$$ 
V(r) \leq - (1 + \delta) \frac{(n-2)^2}{4 r^2}\qquad {\rm for}\quad r \geq R_1.  
$$ 
For $\varphi(r) = \chi(r) r^{\frac{1}{2}}$ defined in Proof of Lemma~\ref{Sub}, 
its unitary transformation is 
$$ 
U^{-1} \varphi(r) = h^{-\frac{n-1}{2}} \varphi(r) 
= r^{-\frac{n-1}{2}} \chi(r) r^{\frac{1}{2}} = \chi(r) r^{-\frac{n-2}{2}}. 
$$ 
Setting $\phi(x) := U^{-1} \varphi(r(x)) = \chi(r(x)) r(x)^{-\frac{n-2}{2}}$ for $x \in \mathbb{R}^n$, 
we then get, by Lemma~\ref{Sub},  
\begin{align*} 
\int_{\mathbb{R}^n} \{ |\nabla \phi|^2 + V \phi^2 \} dx 
& = \sigma_{n-1} \int_0^{\infty} \Big{\{} |\varphi'|^2 + \Big{(} V + \frac{(n-1)(n-3)}{4 t^2} \Big{)} \varphi^2 \Big{\}} dt \\ 
& \leq \sigma_{n-1} \Big{\{} 3 - \frac{(n-2)^2}{4} \delta \log\Big{(}\frac{k}{2}\Big{)} \Big{\}} < 0\qquad 
{\rm if}\quad k \geq k_0. 
\end{align*} 
For each $i \geq 1$, by setting also a function $\phi_i(x) := U^{-1} \varphi_i(r(x))$ in $W^{1,2}(\mathbb{R}^n)$ with compact support, 
then the sequence $\{\phi_i\}_{i=1}^{\infty}$ satisfies the following: 
\begin{align*} 
&\ \ {\rm supp}~\phi_i~\cap~{\rm supp}~\phi_j = \emptyset\qquad {\rm if}\quad i \ne j, \\ 
&\int_{\mathbb{R}^n} \{ |\nabla \phi_i|^2 + V \phi_i^2 \} dx < 0\qquad {\rm for\ all}\quad i \geq 1. 
\end{align*} 
Here, $\varphi_i$ is the function defined in Proof of Lemma~\ref{Sub}. 
Therefore, the min-max principle implies that $\sigma_{\textrm{disc}}(- \Delta + V)$ is infinite. \qed \\ 
\quad \\ 
{\it Proof of Theorem~\ref{Thm-H}}.\ \ \underline{Assertion~(i)}. 
\ \ By using Theorem~\ref{Thm-UPLH}, 
the proof of the finiteness assertion~(i) is similar to that of Theorem~\ref{Thm-RSK}-(i). 
So we omit it. \\ 
\underline{Assertion~(ii)}. 
\ \ Without loss of generality, 
we may assume that $\kappa = 1$. 
First note that, 
the assertion-(ii) is equivalent to that 
$\sigma_{\rm disc}(- \Delta_{g_1} + V - \frac{(n-1)^2}{4})$ is infinite. 
Keeping the same notations as those in Proof of Lemma~\ref{Sub}, 
we will construct a test function on $\mathbb{H}^n(- 1)$. 
Similarly to the argument in Proof of Theorem~RSK-(ii), 
we may also assume that the potential $V$ is a radial function on $\mathbb{H}^n(- 1)$. 

Since $\mathbb{H}^n := (\mathbb{H}^n(- 1), g_1)$ is given by (13) with $h(r) := \sinh r$, 
the potential term of $L_{g_1}$ is 
\begin{align*}  
\frac{(n-1)(n-3)}{4} \Big{(}\frac{h'}{h}\Big{)}^2 + \frac{n-1}{2} \frac{h''}{h} 
&= \frac{(n-1)(n-3)}{4} \coth^2 r + \frac{n-1}{2}, \\  
&= \frac{(n-1)(n-3)}{4} \Big{(} 1 + \frac{1}{\sinh^2 r} \Big{)} + \frac{n-1}{2}, \\ 
&= \frac{(n-1)^2}{4} + \frac{(n-1)(n-3)}{4 \sinh^2 r}, 
\end{align*} 
and hence the potential of the operator 
$U \circ \Big{(} (- \Delta_{g_1})|_{\rm radial} + V(r) - \frac{(n-1)^2}{4} \Big{)} \circ U^{-1}$ 
is given by  
$$ 
\frac{(n-1)(n-3)}{4 \sinh^2 r} + V(r).  
$$ 
If we set $\widetilde{\delta} := \frac{\delta}{2} > 0$ in Lemma~\ref{Sub}, 
the condition for the potential $V(r)$ 
$$ 
V(r) \leq - (1 + \delta) \frac{1}{4 r^2}\qquad {\rm for}\quad r \geq R_1   
$$ 
implies that for the potential of the operator 
$U \circ \Big{(} (- \Delta_{g_1})|_{\rm radial} + V(r) - \frac{(n-1)^2}{4} \Big{)} \circ U^{-1}$ 
$$ 
\frac{(n-1)(n-3)}{4 \sinh^2 r} + V(r) \leq - (1 + \widetilde{\delta}) \frac{1}{4 r^2}\qquad {\rm for}\quad r \geq \widetilde{R},  
$$ 
where $\widetilde{R} = \widetilde{R}(\delta, R_1, n) > 0$ is a large constant. 

For $\varphi(r) = \chi(r) r^{\frac{1}{2}}$ defined in Proof of Lemma~\ref{Sub}, 
its unitary transformation is 
$$ 
U^{-1} \varphi(r) = h^{-\frac{n-1}{2}} \varphi(r) 
= \chi(r) r^{\frac{1}{2}} (\sinh r)^{-\frac{n-1}{2}}. 
$$ 
Setting $\phi(x) := U^{-1} \varphi(r(x)) = \chi(r(x)) r(x)^{\frac{1}{2}} (\sinh r(x))^{-\frac{n-1}{2}}$ 
for $x \in \mathbb{H}^n$, 
we then get 
$$ 
\int_{\mathbb{H}^n} \{ |\nabla \phi|^2 + V \phi^2 \} dv_{g_1}\\ 
\leq \sigma_{n-1} \Big{\{} 3 - \frac{\delta}{8} \log\Big{(}\frac{k}{2}\Big{)} \Big{\}} < 0\qquad 
{\rm if}\quad k \geq k_0. 
$$ 
For each $i \geq 1$, by setting also a function $\phi_i(x) := U^{-1} \varphi_i(r(x))$ in $W^{1,2}(\mathbb{H}^n)$ with compact support, 
then the sequence $\{\phi_i\}_{i=1}^{\infty}$ satisfies the following: 
\begin{align*} 
&\ \ {\rm supp}~\phi_i~\cap~{\rm supp}~\phi_j = \emptyset\qquad {\rm if}\quad i \ne j, \\ 
&\int_{\mathbb{H}^n} \{ |\nabla \phi_i|^2 + V \phi_i^2 \} dv_{g_1} < 0\qquad {\rm for\ all}\quad i \geq 1. 
\end{align*} 
Here, $\varphi_i$ is the function defined in Proof of Lemma~\ref{Sub}. 
Therefore, the min-max principle implies that 
$\sigma_{\textrm{disc}}(- \Delta_{g_1} + V - \frac{(n-1)^2}{4})$ is infinite. \qed \\ 
\quad \\ 
{\it Proof of Theorem~\ref{Thm-ALE}}. 
\ \ As mentioned before in Remark~1.7, 
$(M, g)$ is ALE of order $\tau~(0 < \tau < 1)$ with finitely many ends. 
In order to prove the assertions, it is enough to treat only the case that 
$(M, g)$ has only one end. 
Hence, there exists a relatively compact open set $\mathcal{O}$ such that 
$M - \mathcal{O}$ is connected and that it has {\it coordinates at infinity} 
$x = (x^1, \cdots, x^n)$; that is, 
there exist $R > 0$ and a finite subgroup $\Gamma \subset O(n)$ 
acting freely on $\mathbb{R}^n - B_R({\bf 0})$ 
such that $M - \mathcal{O}$ is diffeomorphic to $\big{(} \mathbb{R}^n - B_R({\bf 0}) \big{)} \big{/} \Gamma$ 
and, with respect to $x = (x^1, \cdots, x^n)$, the metric $g$ satisfies the following on $M - \mathcal{O}$:  
\begin{equation} 
g_{ij} = \delta_{ij} + O(|x|^{-\tau}),\qquad 
\partial_k g_{ij} = O(|x|^{-1-\tau}),\qquad 
\partial_k \partial_{\ell}g_{ij} = O(|x|^{-2-\tau}).  
\end{equation} 
It should be remarked that, even if $R > 0$ is sufficiently large, 
neither the outward normal exponential map for $\{ x \in M - \mathcal{O}\ \big{|}\ |x| \geq R \}$ 
nor the one for $\{ p \in M\ \big{|}\ r(p) \geq R \}$ (regarding each as the set $E$ in Theorem~\ref{Thm-UPLG}) 
may be diffeomorphism. 
So, Theorem~\ref{Thm-UPLG} can not be applied directly to ALE manifolds. 
However, we use 
$$ 
\rho(x) := |x| = \sqrt{\Sigma_{i=1}^{n} (x^i)^2}\qquad {\rm for}\quad x \in M - \mathcal{O}
$$ 
instead of $r(x) = {\rm dist}_g(x, p_0)$ in Proof of Theorem~\ref{Thm-UPLG}, 
and then we can apply this theorem to ALE manifolds in a sense of approximation. 

In terms of the polar coordinates $(\rho, y^2, \cdots, y^n)$ 
arising from coordinates at infinity $(x^1, \cdots, x^n)$, 
the asymptotic behavior (14) implies that the metric $g$ is given as the form on $M - \mathcal{O}$  
\begin{align}  
g(\rho, y) &= g_{\rho \rho}(\rho, y)d\rho^2 + g_{\alpha \beta}(\rho, y)dy^{\alpha}dy^{\beta} 
+ g_{\rho \beta}(\rho, \alpha)d\rho dy^{\alpha}\\ 
&= \ \big{(}1 + O(\rho^{-\tau})\big{)}d\rho^2 + \rho^2\big{(}g_{S^{n-1}(1)} + O(\rho^{-\tau})\big{)} + \rho~O(\rho^{-\tau})d\rho dy. \notag  
\end{align} 
From (14) also, 
there exists sufficiently large $R_2 \geq 2\cdot \max \{R_0, R_1\} > 0$ such that 
\begin{align} 
&\frac{1}{2} \rho(x) \leq r(x) \leq \frac{3}{2} \rho(x)\qquad {\rm for}\quad \rho(x) \geq R_2, \\ 
&r(x) = \rho(x) \big{(} 1 + O(\rho(x)^{-\tau}) \big{)}\qquad {\rm as}\quad \rho(x) \nearrow \infty. 
\end{align} \\ 
\underline{Assertion~(i)}. 
\ \ Set $\sqrt{g} := \sqrt{{\rm det}(g_{\alpha \beta}(\rho, y))}$. 
By (14)--(17), we get 
\begin{align*}  
&dv_g = \big{(} 1 + O(\rho^{-\tau}) \big{)} \sqrt{g}~d\rho~dy^2 \cdots dy^n, \\ 
&\Delta_g \rho = \frac{\partial_{\rho} \sqrt{g}}{\sqrt{g}} + O(\rho^{-1-\tau}),\qquad  
- \partial_{\rho} (\Delta_g \rho) = |\nabla d\rho|^2 + {\rm Ric}_g(\nabla \rho, \nabla \rho) + O(\rho^{-2-\tau}),
\end{align*}  
and hence 
\begin{align} 
&\partial_{\rho}^2 \sqrt{g} = - \{|\nabla d\rho|^2 + {\rm Ric}_g(\nabla \rho, \nabla \rho)\} \sqrt{g} 
+ (\Delta_g \rho)^2 \sqrt{g} + \sqrt{g}~O(\rho^{-2-\tau}), \\ 
&\frac{(\partial_{\rho} \sqrt{g})^2}{\sqrt{g}} = 
(\Delta_g \rho)^2 \sqrt{g} + \sqrt{g}~O(\rho^{-2-\tau}), \\ 
&\Delta _g \rho = \frac{n-1}{\rho} + O(\rho^{-1-\tau}), \\ 
&|\nabla d\rho|^2 = \frac{n-1}{\rho^2} + O(\rho^{-2-\tau}), \\ 
&{\rm Ric}_g(\nabla \rho, \nabla \rho) = O(\rho^{-2-\tau}). 
\end{align} 

With these understandings, modifying Proof of Theorem~\ref{Thm-UPLG}-(i) combined with (14)--(22), 
we have the following: 
For $R \geq R_2$ and $u \in C_c^{\infty}(M)$, 
\begin{align*} 
&\ \int_R^{\infty} |\nabla u|^2 \sqrt{g}d\rho \geq \big{(} 1 + O(R^{-\tau}) \big{)} \int_R^{\infty} |\partial_{\rho} u|^2 \sqrt{g}d\rho \\ 
= & \big{(} 1 + O(R^{-\tau}) \big{)} \int_R^{\infty} \Big{\{} \frac{1}{4 \rho^2} + \frac{1}{4} (\Delta_g \rho)^2 - \frac{1}{2} |\nabla d\rho|^2 
- \frac{1}{2}{\rm Ric}_g(\nabla \rho, \nabla \rho) + O(\rho^{-2-\tau}) \Big{\}} u^2 \sqrt{g} d\rho \\
& \qquad \qquad \ \ + \big{(} 1 + O(R^{-\tau}) \big{)} \frac{1}{2} \Big{(} (\Delta_g\rho - \frac{1}{R}) u^2 \sqrt{g} \Big{)}\Big{|}_{\rho = R} \\ 
\geq & \big{(} 1 + O(R^{-\tau}) \big{)} \int_R^{\infty} \Big{\{} \frac{(n-2)^2}{4 \rho^2} + O(\rho^{-2-\tau}) \Big{\}} u^2 \sqrt{g} d\rho \\ 
& \qquad \qquad \ \ + \big{(} 1 + O(R^{-\tau}) \big{)} \Big{(} \frac{n-2}{2R}  + O(R^{-1-\tau}) \Big{)} \big{(} u^2 \sqrt{g} \big{)}|_{\rho = R},  
\end{align*} 
and hence 
\begin{equation} 
\int_{\{\rho(x) \geq R\}} |\nabla u|^2 dv_g  
\geq \big{(} 1 + O(R^{-\tau}) \big{)} \int_{\{\rho(x) \geq R\}} \Big{\{} \frac{(n-2)^2}{4 \rho^2} + O(\rho^{-2-\tau}) \Big{\}} u^2 dv_g.   
\end{equation} 
Here, we use the assumption that $n \geq 3$. 
Modifying Proof of Theorem~RSK-(i) combined with (16), (17) and (23), 
we can choose large $R \geq R_2$ such that  
$$ 
\int_{\{\rho(x) \geq R\}} \{|\nabla u|^2 + V u^2\}dv_g \geq 0\qquad {\rm for}\quad u \in C_c^{\infty}(M).  
$$ 
This combined with the argument in Proof of Theorem~RSK-(i) completes the proof of Assertion~(i). \\ 
\underline{Assertion~(ii)}. 
\ \ We also modify Proof of Theorem~RSK-(ii) as below. 
Similarly to that, set 
$$ 
\phi(x) := 
\begin{cases} 
\ \ 0\quad &{\rm if}\quad x \in \mathcal{O}, \\ 
\ \ U^{-1}\varphi(\rho(x)) = \chi(\rho(x)) \rho(x)^{-\frac{n-2}{2}} \quad &{\rm if}\quad x \in M - \mathcal{O}, 
\end{cases} 
$$ 
where $\varphi, \chi$ and $U$ are same as those defined in Proof of Theorem~RSK-(ii). 
By (14)--(17), for large $R \geq R_2$, we get 
\begin{align*} 
\int_M \{|\nabla \phi|^2 + V \phi^2\} dv_g 
&\ = \int_{M - \mathcal{O}} \big{\{}|\partial_{\rho} \phi|^2 \big{(} 1 + O(\rho^{-\tau}) \big{)} + V \phi^2 \big{\}} 
\big{(} 1 + O(\rho^{-\tau}) \big{)} d\rho d\sigma_g \\ 
&\ \leq \frac{\sigma_{n-1}}{|\Gamma|} \Big{\{} 3 \big{(} 1 + O(R^{-\tau}) \big{)} 
- \frac{(n-2)^2}{4} \delta \Big{(} \log \Big{(}\frac{k}{2}\Big{)} - \frac{1}{\tau} O(R^{-\tau}) \Big{)} \Big{\}}, 
\end{align*} 
where $|\Gamma|$ denotes the order of $\Gamma$. 
Then, there exists sufficiently large $\widetilde{k}_0 = \widetilde{k}_0(\tau, n, R) > 0$ such that 
$$ 
\int_M \{|\nabla \phi|^2 + V \phi^2\} dv_g < 0 \qquad {\rm for}\quad k \geq \widetilde{k}_0. 
$$ 
The rest part is similar to that in Proof of Theorem~RSK-(ii). \qed \\ 
\quad \\ 
{\it Proof of Theorem~\ref{Thm-AH}}.\ \ 
Let us keep the same notations as those in Remark~1.9. 
Without loss of generality, 
we may assume that $\partial \overline{M}$ is connected. 
Since the conformally rescaled metric $\overline{g} = \lambda^2 g$ is a $C^2$ metric on $\overline{M}$, 
the inward normal $\overline{g}$-exponential map 
$\exp_{\partial \overline{M}} : \mathcal{N}^-(\partial \overline{M}) \rightarrow \overline{M}$ 
induces a $C^3$ diffeomorphism 
$$ 
[0, t_0] \times \partial \overline{M} 
\ni (t, q) \mapsto \exp_q (t \nu(q)) \in (\partial \overline{M} \sqcup T_{t_0}). 
$$ 
Here, 
\begin{align*}  
&\mathcal{N}^-(\partial \overline{M}) 
:= \{ v \in T\overline{M}|_{\partial \overline{M}}\ \big{|}\ v\ {\rm is\ inward\ normal\ to}\ \partial \overline{M} \}, \\ 
&T_{t_0} := \{ x \in M\ \big{|}\ {\rm dist}_{\overline{g}}(x, \partial \overline{M}) \leq t_0 \}, 
\end{align*} 
and $t_0 > 0$ is a small constant 
and that $\nu(q)$ denotes the inward unit normal vector at $q \in \partial \overline{M}$ with respect to $\overline{g}$. 
Under this identification combined with $|d\lambda|^2_{\overline{g}} = 1$ on $\partial \overline{M}$, 
the metric $g$ can be written in the form 
\begin{align} 
&g(t, q) = \frac{1}{\lambda(t,q)^2} \overline{g}(t, q) 
= \frac{1}{\lambda(t,q)^2} (dt^2 + \widehat{g}(t, q))\quad {\rm for}\quad (t, q) \in (0, t_0] \times \partial \overline{M},\\ 
&\lambda(t, q) = t + t^2\cdot f(t, q)\qquad {\rm with}\quad 
f \in C^2([0, t_0] \times \partial \overline{M}), 
\end{align}  
where $\{\widehat{g}(t, \bullet)\}_{t \in [0, t_0]}$ is a $C^2$ family of $C^2$ metrics on $\partial \overline{M}$. 
Set 
\begin{equation} 
\rho(t) := - \int_{t_0}^t \frac{ds}{s} = - \log\Big{(} \frac{t}{t_0} \Big{)}\quad 
\Big{(}\ \Leftrightarrow\ \ t(\rho) = t_0 e^{-\rho}\ \Big{)}. 
\end{equation} 
Then, we also get a a $C^3$ diffeomorphism 
$$ 
[0, \infty) \times \partial \overline{M} 
\ni (\rho, q) \mapsto \exp_q (t(\rho) \nu(q)) \in T_{t_0}. 
$$ 
Under these identifications $[0, \infty) \times \partial \overline{M} \cong (0, t_0] \times \partial \overline{M} \cong T_{t_0}$, 
from (24)--(26), the metric $g$ also can be written in the form 
\begin{align} 
&g(\rho, q) = \big{(} 1 + O''(e^{-\rho}) \big{)} d\rho^2 + e^{2\rho}~\widetilde{g}(\rho, q)\qquad 
{\rm for}\quad (\rho, q) \in [0, \infty) \times \partial \overline{M}, \\ 
&\widetilde{g}(\rho, q) := \frac{e^{-2\rho}}{\lambda(\rho, q)^2}~\widehat{g}(\rho, q) 
= \frac{1}{t_0^2} \big{(} 1 + O''(e^{-\rho}) \big{)} \widehat{g}(\rho, q), \notag 
\end{align} 
where $\psi := O''(e^{-\rho})$ means that $\psi = O(e^{-\rho}), \nabla \psi = O(e^{-\rho})$ 
and $\nabla^2 \psi = O(e^{-\rho})$. 
From (24)--(27), we note 
\begin{equation}  
{\rm dist}_g((\rho, q), \partial T_{t_0}) = - \int_{t_0}^t \frac{ds}{\lambda(s, q)} = \rho + O(1)\quad 
{\rm for}\quad (\rho, q) \in [0, \infty) \times \partial \overline{M}. 
\end{equation} 
Without loss of generality, we may assume that $p_0 \in M - T_{t_0}$. 
It then follows from (28) that 
\begin{equation} 
\rho(x) - C \leq r(x) = {\rm dist}_g(x, p_0) \leq \rho(x) + C,\qquad 
{\rm for}\quad x \in T_{t_0} 
\end{equation} 
for some constant $C > 0$. 

With these understandings, we can now give several estimates below for the proofs of the assertions (i), (ii). 
Let $y = (y^2, \cdots, y^n)$ be local coordinates on a coordinate open neighborhood $U$ of $\partial \overline{M}$. 
From (27), 
the metric $g$ can be written 
on $\{ (\rho, y) \in [0, \infty) \times U \} \subset [0, \infty) \times \partial \overline{M}$ in the form 
\begin{equation} 
g(\rho, q) = \big{(} 1 + O''(e^{-\rho}) \big{)} d\rho^2 + e^{2\rho}~\widetilde{g}_{\alpha \beta}(\rho, q)dy^{\alpha}dy^{\beta}
\qquad (2 \leq \alpha, \beta \leq n). 
\end{equation} 
Set $\sqrt{g} := e^{(n-1)\rho} \sqrt{{\rm det}(\widetilde{g}_{\alpha \beta})}$.  
By (24)--(27), 
we get  
\begin{align*} 
&dv_g = \big{(} 1 + O''(e^{-\rho}) \big{)} \sqrt{g}~d\rho~dy^2 \cdots dy^n\qquad 
{\rm on}\quad [0, \infty) \times U, \\ 
&\Delta_g \rho = \frac{\partial_{\rho} \sqrt{g}}{\sqrt{g}} + O'(e^{-\rho}),\qquad 
- \partial_{\rho}(\Delta_g \rho) = |\nabla d\rho|^2 + {\rm Ric}_g(\nabla \rho, \nabla \rho) + O(e^{-\rho}), 
\end{align*}  
where $\varphi := O'(e^{-\rho})$ also means that $\varphi = O(e^{-\rho})$ and $\nabla \varphi = O(e^{-\rho})$. 
By (24)--(27) combined with 
$\partial_{\rho} = t \partial_t = t_0 e^{-\rho} \partial_t$ 
and the curvature estimate in \cite[Proposition]{Ma} (cf.~\cite[Lemma~2.1]{Gr-Le}), 
we also obtain  
\begin{align} 
&\partial_{\rho}^2 \sqrt{g} = - \{|\nabla d\rho|^2 + {\rm Ric}_g(\nabla \rho, \nabla \rho)\} \sqrt{g} 
+ (\Delta_g \rho)^2 \sqrt{g} + \sqrt{g}~O(e^{-\rho}), \\ 
&\frac{(\partial_{\rho} \sqrt{g})^2}{\sqrt{g}} = 
(\Delta_g \rho)^2 \sqrt{g} + \sqrt{g}~O'(e^{-\rho}), \\ 
&\Delta_g \rho = (n - 1) + O'(e^{-\rho}), \\  
&|\nabla d\rho|^2 = (n -1) + O'(e^{-\rho}), \\ 
&{\rm Ric}_g(\nabla \rho, \nabla \rho) = - (n - 1) + O(e^{-\rho}). 
\end{align} 
Modifying Proof of Theorem~\ref{Thm-UPLG}-(i) combined with (24)--(27) and (30)--(35), 
we have the following: 
For sufficiently large $R \geq R_0$ and $u \in C_c^{\infty}(M)$, 
\begin{align*} 
&\ \int_R^{\infty} |\nabla u|^2 \big{(} 1 + O''(e^{-\rho}) \big{)} \sqrt{g}d\rho 
\geq \int_R^{\infty} |\partial_{\rho} u|^2 \big{(} 1 + O''(e^{-\rho}) \big{)} \sqrt{g}d\rho \\ 
= & \int_R^{\infty} \Big{\{} \frac{1}{4 \rho^2} + \frac{1}{4} (\Delta_g \rho)^2 - \frac{1}{2} |\nabla d\rho|^2 
- \frac{1}{2}{\rm Ric}_g(\nabla \rho, \nabla \rho) + O(e^{-\rho}) \Big{\}} \big{(} 1 + O(e^{-\rho}) \big{)} u^2 \sqrt{g} d\rho \\
& \qquad \qquad \qquad \quad + \big{(} 1 + O(e^{-R}) \big{)} \frac{1}{2} \Big{(} (\Delta_g\rho - \frac{1}{R}) u^2 \sqrt{g} \Big{)}\Big{|}_{\rho = R} \\ 
\geq & \int_R^{\infty} \Big{\{} \frac{1}{4 \rho^2} + \frac{(n-1)^2}{4}+ O(e^{-\rho}) \Big{\}} \big{(} 1 + O(e^{-\rho}) \big{)} u^2 \sqrt{g} d\rho \\ 
& \qquad \qquad \qquad \quad + \big{(} 1 + O(e^{-R}) \big{)} \Big{(} (n - 1) - \frac{1}{R}  + O(e^{-R}) \Big{)} \big{(} u^2 \sqrt{g} \big{)}|_{\rho = R}. 
\end{align*} 
Therefore,  
\begin{align} 
&\ \int_{\{\rho(x) \geq R\}} |\nabla u|^2 dv_g \\ 
\geq &\ \int_{\{\rho(x) \geq R\}} \Big{\{} \frac{(n-1)^2}{4} + \frac{1 + O(e^{-\rho})}{4 \rho^2} + O(e^{-\rho}) \Big{\}} u^2 dv_g. \notag   
\end{align} 

By the formulae (24)--(28), (30) and the estimates (29), (31)--(36), 
we can obtain the assertions (i), (ii) similarly to Proof of Theorem~\ref{Thm-ALE}. 
So we omit the rest of the proofs. 
\qed \\ 
\quad \\ 

\section{Further comments} 

In this section, we first give one more application of Thorem~1.1. 
As a matter of convenience, we introduce some terminology below. 
Let $(M, g)$ be a noncompact complete $n$-manifold. 
Assume that there exists an open subset $U$ of $M$ with compact $C^{\infty}$ boundary $\partial U$ 
such that the outward normal exponential map $\exp_{\partial U} : \mathcal{N}^{+}(\partial U) \to M-U$ is a diffeomorphism. 
Set $E = M-U$ and $r(\bullet) = {\rm dist}_g(\bullet, \partial U)$ on $E$. 
A tangent $2$-plane $\pi\subset T_xM$ at $x\in E$ is said to be {\it radial} 
if the plane $\pi$ contains $\nabla r$, 
and that the restriction of the sectional curvatures to all radial planes 
are called the {\it radial curvatures}. 
Using Hessian comparison theorem in Riemannian geometry, we get the following:
\begin{thm}
Let $(M, g)$ be a noncompact complete $n$-manifold with $n \geq 2$. 
Assume that there exists a relatively compact open subset $U$ of $M$ 
with compact $($possibly disconnnected\,$)$ $C^{\infty}$ boundary 
$\partial U$ such that the outward normal exponential map 
$\exp_{\partial U} : \mathcal{N}^{+}(\partial U) \to M - U$ is a diffeomorphism. 
Set $r(\bullet) := {\rm dist}_g(\bullet, \partial U)$ on $M - U$. 
Assume also that 
\begin{equation}  
\nabla dr \geq 0\qquad {\rm on}\quad \partial U, 
\end{equation} 
and the radial curvatures satisfy 
\begin{align} 
&\ ``~{\rm radial~curvatures}~" \leq 0\qquad{\rm on}\quad M - U, \\ 
&\ - (\kappa + \delta_1 r^{-2}) \leq ``~{\rm radial~curvatures}~" \leq - (\kappa + \delta_2 r^{-2})\qquad {\rm for}\quad r \geq R_0 
\end{align} 
for some positive constants $\kappa > 0, R_0 > 0$ and 
some constants $\delta_1, \delta_2$ with $\delta_1 \geq \delta_2$ $($not necessarily positive$)$. 
Moreover, assume that
$$
1 - (2n-5)\delta_1 + (n^2-4)\delta_2 > 0.
$$ 
Then, $\sigma_{{\rm ess}}(-\Delta_g)=[\frac{(n-1)^2\kappa}{4}, \infty)$ and that 
$\sigma_{{\rm disc}}(-\Delta_g)$ is finite. 
\end{thm}
\begin{proof}
Applying Hessian comparison theorem \cite{He-Ka} to $r$ combined with (37) and (38), 
we have 
\begin{equation} 
\nabla dr\Big{|}_{(\nabla r)^{\perp} \times (\nabla r)^{\perp}} \geq 0\qquad {\rm on}\quad M - U,  
\end{equation} 
and hence, by Hessian comparison theorem again combined with (38)--(40), 
$$ 
\Big{\{} \sqrt{\kappa} + \frac{\delta_2}{2\sqrt{\kappa}r^2} + o(r^{-2}) \Big{\}} 
\leq \nabla dr\Big{|}_{(\nabla r)^{\perp} \times (\nabla r)^{\perp}}
\leq \Big{\{} \sqrt{\kappa} + \frac{\delta_1}{2\sqrt{\kappa}r^2} + o(r^{-2}) \Big{\}}\qquad {\rm on}\quad M - U.
$$ 
In particular, $\lim_{r\to \infty}\Delta_g r=(n-1)\sqrt{\kappa}$, and hence 
$\sigma_{{\rm ess}}(-\Delta_g)=[\frac{(n-1)^2\kappa}{4}, \infty)$ 
(see \cite[Theorem~1.2]{Ku} for details). 

For the second assertion, we set  
$$ 
t = t(r) := \displaystyle\sqrt{\kappa} + \frac{\delta_2}{2\sqrt{\kappa} r^2} + o(r^{-2}) \geq 0\quad 
{\rm and}\quad 
T = T(r) := \displaystyle\sqrt{\kappa} + \frac{\delta_1}{2\sqrt{\kappa} r^2} + o(r^{-2}) \geq 0.
$$ 
Then, 
\begin{align*}
(\Delta_g r)^2 - 2|\nabla dr|^2 \geq &\ \left\{T+(n-2)t\right\}^2 - 2\left\{(n-2)T^2 + t^2\right\}\\
= &\ -(2n-5)T^2 + \{n^2-4n+2\}t^2 + 2(n-2)tT\\
\geq &\ -(2n-5)T^2 + \{n^2-4n+2\}t^2 + 2(n-2)t^2\\
= &\ -(2n-5)T^2 + \{n^2-2n-2\}t^2\\ 
= &\ (n-1)(n-3)\kappa + \left\{- (2n-5)\delta_1 + (n^2-2n-2)\delta_2\right\}\frac{1}{r^2} + o(r^{-2}).
\end{align*}
Hence, for $r \geq R_0$, we have 
\begin{align*}
&\ \frac{1}{4}(\Delta_g r)^2 - \frac{1}{2}|\nabla dr|^2 - \frac{1}{2}{\rm Ric}_g(\nabla r, \nabla r) + \frac{1}{4 r^2}\\
\geq &\ \frac{1}{4} \Big{[} (n-1)(n-3)\kappa + \frac{1}{r^2} \left\{- (2n-5)\delta_1 + (n^2-2n-2)\delta_2\right\} + o(r^{-2}) \Big{]}\\
&\ + \frac{n-1}{2} (\kappa + \frac{\delta_2}{r^2}) + \frac{1}{4 r^2}\\
= &\ \frac{(n-1)^2\kappa }{4} + \frac{1}{4 r^2}\left\{1 - (2n-5)\delta_1 + (n^2-4)\delta_2\right\} + o(r^{-2}).
\end{align*}
Since $1 - (2n-5)\delta_1 + (n^2-4)\delta_2 > 0$ and $\Delta_g r = (n-1) \kappa + O(r^{-2})$, 
we obtain the desired finiteness result similarly to Proof of Theorem~RSK. 
\end{proof}

Next, we give an interesting example of a noncompact manifold on which Theorem~1.1 holds. 
In the inequality (2) of Theorem~1.1, 
the term appearing in the integrand on the boundary is not the second fundamental form itself 
but its trace, that is, the mean curvature of the boundary. 
Hence, for example, Theorem~1.1 allows us to consider manifolds whose ends has a mixed structure, {\it expanding} and {\it shrinking}. 
It should be also remarked that the manifold given below 
is not rotationally symmetric, 
but our argument for constructing a nice test function is still applicable.   

Let $\xi$ be a unit Killing vector field on the unit standard $3$-sphere 
$(S^3(1),g_{S^3(1)})$ which satisfies $\ker d\pi = \R \cdot \xi$ for the Hopf fibering 
$\pi:S^3(1)\to \CP ^1$. We define a symmetric tensor $g_h$ on $S^3(1)$ 
by $g_h := g_{S^3(1)} - \omega_{\xi} \otimes \omega_{\xi}$, where $\omega_{\xi}$ stands for 
the $1$-form dual to $\xi$ with respect to $g_{S^3(1)}$. 
Using these two tensors $g_h$ and $\omega_{\xi}\otimes \omega_{\xi}$, we define a 
Riemmanian metric $g_{\mu\nu}$ on $\R^4$ by
$$
g_{\mu\nu} := dr^2 + \mu(r)^2g_h + \nu(r)^2\omega_{\xi} \otimes \omega_{\xi},
$$
where $r$ is the Euclidean distance to the origin ${\bf 0}$ of $\R ^4$, and 
$\mu$ and $\nu$ are smooth functions on $[0,\infty)$ satisfying
\begin{align*}
&\mu(0) = \nu(0) = 0, \quad \mu'(0) = \nu'(0) = 1,\\
&\mu > 0,\quad \nu > 0\qquad \mbox{on}\quad (0,\infty).
\end{align*}
We also choose $\mu$ and $\nu$ satisfying  
$$
\mu(r) = e^{r}, \quad \nu(r) = e^{-r}\qquad \mbox{for}\quad r \geq R_0, 
$$
where $R_0 > 0$ is some positive constant. 
\begin{prop}\label{Example}\ \ 
With these understandings, assume that $(- \Delta_{g_{\mu\nu}} + V)|_{C^{\infty}_c(\mathbb{R}^4)}$ is essentially self-adjoint 
on $L^2(\mathbb{R}^4, dv_{g_{\mu\nu}})$, 
and that $\sigma_{{\rm ess}}(- \Delta_{g_{\mu\nu}} + V) = [\frac{1}{4}, \infty)$. \\ 
$({\rm i})$\quad Assume that, there exists $R_0 > 0$ such that $V$ satisfies 
$$ 
V(x) \geq - \frac{1}{4r^2}\qquad {\rm for}\quad r \geq R_0. 
$$ 
Then, $\sigma_{{\rm disc}}(- \Delta + V)$ is finite. \\ 
$({\rm ii})$\quad Assume that, there exist $\delta > 0$ and $R_1 > 0$ such that $V$ satisfies 
$$ 
V(x) \leq - (1+\delta)\frac{1}{4r^2}\qquad {\rm for}\quad r \geq R_1. 
$$ 
Then, $\sigma_{{\rm disc}}(- \Delta + V)$ is infinite.  
\end{prop} 
\begin{proof} 
Under our assumption, we have for $r \geq R_0$
$$
\Delta_{g_{\mu\nu}} r = 1,\quad |\nabla dr|^2 = 3,\quad {\rm Ric}_{g_{\mu\nu}}(\nabla r,\nabla r) = -3.
$$
Hence, we see that $\sigma_{{\rm ess}}(-\Delta _{g_{\mu\nu}})=[\frac{1}{4}, \infty)$, 
and that the inequality $(2)$ in Theorem~$1.1$ becomes the following simple form:
\begin{align*}
&\ \int_{\R ^4 - B_R({\bf 0})} \Big{(} |\nabla u|^2 - \frac{1}{4} u^2 \Big{)}dv_{g_{\mu\nu}}\\
\geq &\ \int_{\R ^4-B_R({\bf 0})} \Big{\{} \frac{1}{4 r^2} + \frac{1}{4}(\Delta_{g_{\mu\nu}} r)^2 - \frac{1}{2}|\nabla d r|^2 
- \frac{1}{2} {\rm Ric}_{g_{\mu\nu}}(\nabla r, \nabla r) - \frac{1}{4} \Big{\}} u^2 dv_{g_{\mu\nu}} \\
&\ + \frac{1}{2} \int_{\partial B_R({\bf 0})} \Big{(} \Delta_{g_{\mu\nu}} r - \frac{1}{R} \Big{)} u^2 d\sigma_{g_{\mu\nu}}\\
= &\ \int_{\R ^4 - B_R({\bf 0})} \frac{1}{4 r^2} u^2 dv_{g_{\mu\nu}} + \frac{1}{2} \int_{\partial B_R({\bf 0})} 
\Big{(} 1 - \frac{1}{R} \Big{)} u^2 d\sigma_{g_{\mu\nu}}\\
\geq &\ \int_{\R ^4 - B_R({\bf 0})} \frac{1}{4 r^2} u^2 dv_{g_{\mu\nu}},
\end{align*}
when $R \geq \max \{1, R_0\}$.
Therefore, we get our desired assertion~(i) as in Proof of Theorem~\ref{Thm-H}. 
On the other hand, if we choose our test function $\varphi$ as 
$$ 
\varphi = \chi(r) r^{\frac{1}{2}} e^{-\frac{r}{2}}, 
$$ 
we also get our desired assertion~(ii) as in Proof of Theorem~\ref{Thm-H}. 
Here, $\chi(r)$ is same as that defined in Proof of Lemma~\ref{Sub}. 
\end{proof}

\newpage 

\bibliographystyle{amsbook}

\vspace{20mm} 

\end{document}